\documentclass[a4paper, 11pt, reqno]{amsart}
\usepackage[english]{babel}
\usepackage[utf8]{inputenc}
\usepackage{amsthm, amsmath, amssymb, amsrefs, enumerate, enumitem, mathtools, mathrsfs}

\usepackage{a4wide}

\usepackage[T1]{fontenc}
\usepackage{lmodern}
\usepackage{microtype}
\usepackage{xcolor}

\usepackage{hyperref}

\usepackage{enumitem}

\theoremstyle{plain}
\newtheorem{thm}{Theorem}[section]
\newtheorem{prop}[thm]{Proposition}
\newtheorem{lemma}[thm]{Lemma}
\newtheorem{cor}[thm]{Corollary}
\theoremstyle{definition}
\newtheorem{defn}[thm]{Definition}
\newtheorem*{exa}{Example}
\theoremstyle{remark}
\newtheorem*{rmk}{Remark}

\newcommand{\C}{\mathbb{C}}
\newcommand{\Hb}{\mathbb{H}}
\newcommand{\Z}{\mathbb{Z}}

\newcommand{\N}{\mathbb{N}}

\newcommand{\slz}{{\text {\rm SL}}_2(\mathbb{Z})}

\newcommand{\Hs}{\mathscr{H}}

\numberwithin{equation}{section}
\setlist{nosep}
\setlist{noitemsep}

\makeatletter
\let\@@pmod\pmod
\DeclareRobustCommand{\pmod}{\@ifstar\@pmods\@@pmod}
\def\@pmods#1{\mkern4mu({\operator@font mod}\mkern 6mu#1)}
\makeatother

\title{Relations between higher level Hurwitz class numbers}
\author{Ngoc Trinh Le}
\address{Department of Mathematics, Tulane University, New Orleans, LA 70118}
\email{tle20@tulane.edu}

\begin{document}

\maketitle

\begin{abstract}
    We connect generalizations of the classical Hurwitz class numbers coming from two different frameworks: one introduced by Pei and Wang, arising from the generalized Cohen--Eisenstein series, and another by Li, Skoruppa, and Zhou, arising from Eichler orders of quaternion algebras. As applications, we obtain new basis for Eisenstein space $E_{3/2}^{+}(4N,\mathrm{id})$, a generalization of recent results of Beckwith and Mono, and a generalization of Gauss' formula.
\end{abstract}

\section{Introduction}
The study of the connections between ternary quadratic forms and class numbers dates back to Gauss's formula for the number of ways of writing an integer $n$ as a sum of three squares. In terms of the Hurwitz class numbers $H(n)$, Gauss' formula can be expressed as 
\begin{equation}\label{eq:threesquares}
\left( \sum_{n \in \mathbb{Z}} q^{n^2} \right)^3= 12 \sum_{ n \ge 0} (H(4n) - 2 H(n)) q^n.
\end{equation}

We prove a level $p$ analogue of \eqref{eq:threesquares}. We note that \eqref{eq:threesquares} was also generalized by Bringmann and Kane \cite{BringmannKane}, who related several ternary theta functions to combinations of Hurwitz class numbers. In a different direction, we replace $H(n)$ with a higher level analogue. To explain this, we let 
$$
Q_{N,n}\coloneqq \left\{ ax^2 + bxy + cy^2 : a,b,c \in \mathbb{Z},\ N \mid a,\ b^2 - 4ac = n \right\}
$$ and $$
\Gamma_0(N)
\coloneqq
\left\{
\begin{pmatrix}
a & b \\
c & d
\end{pmatrix}
\in \mathrm{SL}_2(\mathbb{Z})
\;\middle|\;
c \equiv 0 \mod N
\right\}.
$$
Then $\Gamma_0(N)$ acts on $Q_{N,n}$ via linear substitution. We recall that the Hurwitz class numbers $H(n)$ are defined in terms of the $\operatorname{SL}_2(\mathbb{Z})$ action on $\mathcal{Q}_{1,-n}$. From this perspective, the quantity $$\displaystyle\sum_{Q \in \mathcal{Q}_{p,-n}/\Gamma_0(p)} \frac{2}{|\Gamma_0(p)_Q|}$$  can be viewed as a natural generalization of the Hurwitz class number $H(n)$, where $\Gamma_0(p)_Q$ is the stabilizer of $Q \in Q_{N,n}$ under the $\Gamma_0(p)$ action. Indeed, when $p=1$, this quantity is equal to $H(n)$. Using Theorem \ref{thm:BM4} and \ref{thm:PWtoLZ}, we obtain the following analog of \ref{eq:threesquares}:
\begin{thm}\label{thm:Main4}
Let $p\neq 2$ be a prime number. If $n>0$ with $n \equiv 0,3 \mod 4$, then we have
$$
12\sum_{Q \in \mathcal{Q}_{p,-4n}/\Gamma_0(p)} \frac{2}{|\Gamma_0(p)_Q|}
-
24\sum_{Q \in \mathcal{Q}_{p,-n}/\Gamma_0(p)} \frac{2}{|\Gamma_0(p)_Q|}
=
4r_3(n)
-
2\big(r_3(p^2 n) - p r_3(n)\big)
$$

where $\displaystyle\sum _{n \ge0}r_3(n)q^n\coloneqq\left( \sum_{n \in \mathbb{Z}} q^{n^2} \right)^3$.
\end{thm}

\begin{exa}
When $n=35$ and $p=31$, there are no forms $\mathcal{Q}_{p,-4n}$ or $\mathcal{Q}_{p,-n}$, since $\left( \frac{-35}{31} \right)=-1$ (where $\left(\frac{c}{d}\right)$ denotes the Kronecker symbol). Thus the left-hand-side is $0$. On the other hand, we compute $r_3(35)$ and $r_3(41^2 35) = 2064$, and can check that $4r_3(n)
-
2\big(r_3(p^2 n) - p r_3(n)\big) = 0$.    
\end{exa}
\begin{exa}
Set $n=43$ and $p=31$. Then we compute $r_3(n)=24$ and $r_3(43 \cdot 31^2) = 744$. Hence the right-hand-side of the equation in Theorem \ref{thm:Main4} is 
$$
4 \cdot 24
-
2\big(744 - 31 \cdot 24\big) = 96.
$$
On the other hand, letting $[a,b,c]$ denote $ax^2 + bxy + cy^2$, we compute
$$
Q_{31,-43}/\Gamma_0(31) = \{ [181, 89, 11],
[7463, 573, 11],
[-181, -89, -11],
[-7463, -573, -11] \}
$$
while $Q_{31,-4\cdot43}$ contains the 16 classes
\begin{align*}
&[3484, 774, 43],
[20813, 1892, 43],
[-3484, -774, -43],
[-20813, -1892, -43],\\
&[362, 178, 22],
[14926, 1146, 22],
[-362, -178, -22],
[-14926, -1146, -22],\\
&[3532, 394, 11],
[9197, 636, 11],
[-3532, -394, -11],
[-9197, -636, -11],\\
&[52, 46, 11],
[1889, 288, 11],
[-52, -46, -11],
[-1889, -288, -11].
\end{align*}
The stabilizers for all of these classes are trivial, since the associated fundamental discriminants are less than $-4$. Hence the left-hand-side of Theorem \ref{thm:Main4} is $12 \cdot 16 - 24 \cdot 4 = 12\cdot 8 = 96 $.
\end{exa}

There are many other higher level generalizations of Hurwitz class numbers, including those with natural connections to modular forms. Equation \eqref{eq:threesquares} hints at the modularity of the generating function of $H(n)$, since, if we let $q \coloneqq e^{2\pi i \tau}$, with $ \tau \in \Hb \coloneqq \left\{\tau = u+iv \colon v > 0\right\}$, we may interpret the left hand side of \eqref{eq:threesquares} as a modular form of weight $\frac{3}{2}$. Zagier \cite{Zagier1} showed that the $H(n)$ generating function itself is a mock modular form (a holomorphic function that becomes modular after adding a specific non-holomorphic correction) of weight $\frac{3}{2}$. This understanding allows one to prove various recurrence formulas for the Hurwitz class numbers, such as the well-known Kronecker-Hurwitz class number relations \cite{Kronecker}
    \begin{equation}
    \label{eq:KH}
    \sum_{r\in\mathbb Z}H(4n-r^2)+\lambda_1(n)=2\sigma_1(n),
    \end{equation}
which were generalized by Mertens \cite{Mertens} \cite{Mertens16}, and refined recently by Ortiz, Raum and Richter \cite{ORR} (where $\lambda_s(n) \coloneqq \displaystyle\sum_{ab=n}\max\{a,b\}^s$ and $\sigma_s(n)\coloneqq\displaystyle\sum_{d|n}d^s$).

 Higher weight analogues of $H(n)$ were introduced by Cohen \cite{CohenClassNumbers}, providing a family of weight $\frac{k}{2}$ Eisenstein series for odd $k \ge 5$ whose coefficients generalize $H(n)$.  The generalized Hurwitz class numbers defined by Pei and Wang \cite{pei-wang} extended the results of \cite{Cohen} to $\Gamma_0(N)$ forms for weights greater than $\frac{3}{2}$ and to a holomorphic weight $\frac{3}{2}$. The coefficients, denoted $H_{m,N}(n)$ where $m| N$ (see Section \ref{sec:Eisenstein} for notation), extend the classical Cohen–Eisenstein coefficients by expressing them in a more general form involving special values of $L$-functions and divisor sums. These were recently studied by Beckwith and Mono (\cite{Beckwith-Mono-1}, \cite{Beckwith-Mono-2}), who explicitly related them to so-called sesquiharmonic Maass forms and certain Siegel theta lifts building on seminal work by Duke-Imamo\=glu-T\'{o}th \cite{DIT} and Bruinier-Funke-Imamo\=glu \cite{BFI}.

At the end of Section 1 of \cite{Beckwith-Mono-1}, it w noted that when $N = 5$ or $N=7$, the generating function $\Hs_{m,N}(\tau):=\sum_{n =0}^{\infty} H_{m,N}(n) q^n$ is a multiple of a theta function for a positive definite ternary quadratic form. More specifically, for \begin{align*}
Q_5(x,y,z) &\coloneqq 7x^2 + 3 y^2 + 7 z^2 + 2 xy - 6 xz + 2 yz, \\
Q_7(x,y,z) &\coloneqq 4x^2 + 7 y^2 + 8 z^2 - 4 xz,
\end{align*}
we have
\begin{align}\label{eq:Ex5}
3\Hs_{5,5}(\tau) &= -12\Big(\sum_{n \geq 0} H_{1,1}(n)q^n - \frac{6}{5} \sum_{n \geq 1} H_{1,5}(n)q^n\Big) = \sum_{(a,b,c) \in \Z^3} q^{Q_5(a,b,c)} \\
&= 1 + 2q^3 + 6q^7 + 6q^8 + 8q^{12} + 6q^{15} + O\left(q^{20}\right), \nonumber
\end{align}
and
\begin{align}\label{eq:Ex7}
2\Hs_{7,7}(\tau) &= -12\Big(\sum_{n \geq 0} H_{1,1}(n)q^n - \frac{8}{7} \sum_{n \geq 1} H_{1,7}(n)q^n\Big) = \sum_{(a,b,c) \in \Z^3} q^{Q_7(a,b,c)} \\
&= 1 + 2q^4 + 2q^7 + 4q^8 + 4q^{11} + 8q^{15} + 6q^{16} + O\left(q^{20}\right). \nonumber
\end{align}
This paper will explain how these examples are part of a broader framework.

The technique is to relate the class numbers of \cite{pei-wang} to a different type of generalized class number defined by Li, Skoruppa and Zhou  \cite{LiSkoruppaZhou} . Their ``modified class numbers'' $H^{(N_1,N_2)}(n)$ (see Section \ref{sec:theta-functions} for notations) appear as Fourier coefficients of a weighted sum of Jacobi theta series associated to Eichler orders (here $N_1$, $N_2$ are two odd coprime square-free integers). This comes from a connection with a famous construction of Gross \cite{Gross} of ternary theta functions using maximal orders of quaternion algebras, which has been extended in several papers (e.g. \cite{BoylanSkoruppa}, \cite{li2026}). 

Using the connection between Eichler orders and ternary quadratic forms, Luo and Zhou \cite{LuoZhou} expressed the class numbers of \cite{LiSkoruppaZhou} in terms of ternary quadratic forms.
 Using their formulas for modified class numbers, we can express the values $H_{m,N}(n)$ in terms of representation numbers for ternary quadratic forms. 

Let $N^{odd}$ be any divisor of $N$ with an odd number of prime factors and $G_{4N, 16 N^2, N^{odd}}$ be the sets of equivalence classes of positive definite ternary quadratic forms of level $N$ and discriminant $16N^2$ that are anisotropic at primes $p$ (i.e., the forms have no non-trivial zero in $\mathbb{Q}_p$) only for $p | N^{odd}$. 

\begin{thm}\label{thm:Main}
    Let $N>1$ be an odd square-free integer. Let $m > 1$ be a divisor of $N$ and let $q_{mg}$ be a prime divisor of $m g$ for each $g | \frac{N}{m}$. For all $n \ge 0$, we have
    \begin{equation}\label{eq:Main}
          \frac{2^{\omega(\frac{N}{m})}m}{N} \Hs_{m,N}(\tau)= 2^{\omega(N) + 1} \sum_{g \mid \frac{N}{m}} u(g) T_{mg}(\tau)
    \end{equation}
    where  $\omega(N), u(g)$ are constants (see Theorem \ref{thm:PWtoLZ} for notations),
    \begin{align*}
   & T_{mg}(\tau) =
    &\begin{cases}
    \displaystyle  \sum_{Q \in G_{4N, 16N^2, mg}} \frac{\theta_Q(\tau)}{|\operatorname{Aut}(Q)|} &  \text{if } \mu(mg) = -1 \\[4ex]
    \displaystyle 2\sum_{Q \in G_{4 \frac{N}{q_{mg}}, 16
    \big(\frac{N}{q_{mg}}\big)^2,\frac{mg}{q_{mg}}}} \frac{ \theta_Q(\tau)}{|\operatorname{Aut}(Q)|} - \sum_{Q \in G_{4N, 16N^2,\frac{mg}{q_{mg}}}} \frac{\theta_Q(\tau)}{|\operatorname{Aut}(Q)|} & \text{if } \mu(m g) = 1,
   \end{cases}
    \end{align*}
    and $\theta_Q(\tau)=\displaystyle\sum_{(a,b,c)\in\mathbb{Z}^3}q^{ Q(a,b,c)}$.
\end{thm}
 
\begin{exa}
    The examples \eqref{eq:Ex5} and \eqref{eq:Ex7} correspond to cases where $m=N=5$ (and $m=N=7$ respectively). In those cases, the right hand side of \ref{eq:Main} contains a single genus containing exactly one equivalence class. Representatives for these equivalence classes are given by $Q_5$ and $Q_7$ respectively. (See Section \ref{sec:Maintm} for examples where $m=N=11$ and $m=N=35$).
\end{exa}

On the other hand, by the Siegel--Weil formula (\cite{Weil1964},  \cite{Weil1965}), the Siegel weighted genus average is identified with an Eisenstein series of weight $3/2$. Restricting to Kohnen's plus subspace, one obtains a basis for the Eisenstein subspace $E^{+}_{3/2}(4N,\mathrm{id})$ (see Section \ref{sec:halfintegral} for notation) in terms of these Siegel weighted genus averages, as stated in the following theorem. In particular, the examples \eqref{eq:Ex5} and \eqref{eq:Ex7} arise as special cases of this proposition.
\begin{prop}\label{thm:main2}
    Let $N$ be an odd square-free integer. Then the set $\{T_{m,N}(\tau)\}_{\substack{m \mid N \\ m \neq 1}}$ is a basis of $E^{+}_{3/2}(4N,\mathrm{id})$. In particular, let $p >2$ be a prime number. Then
\[
E^{+}_{3/2}(4p,\mathrm{id})
=
\left\langle
T_{p,p}(\tau)
=
\frac{1}{2^{\omega(N)+1}}\Hs_{p,p}(\tau)
\right\rangle.
\]
Moreover, for all ternary quadratic forms $Q(x,y,z)$ of level $4p$ and discriminant $(4p)^2$, we have that
\[
\theta_Q (\tau) - \frac{1}{L_p(-1,\mathrm{id})} \sum_{n =0}^{\infty} H_{p,p}(n) q^n
\]
is a weight $\frac{3}{2}$ cusp form of level $4p$, where $L_p(-1,\mathrm{id})$ is the central $L$-value (see Section~\ref{sec:Eisenstein} for notation).
\end{prop}

\begin{exa}
    Consider $p=3, 5,7, 13$. The cusp form in Proposition \ref{thm:main2} is 0 simply because $S^+_{3/2}(4p)=\{0\}$ for these values of $p$.
\end{exa}
\begin{exa}
Consider $p=11$. Sage produces the following quadratic form of level $4\cdot 11$ and $4^2\cdot11^2$:
$$
Q(x,y,z)=3 x^2 -  2xy - 2xz + 15 y^2 +15 z^2 - 14 yz
$$
Computing the representation numbers using Sage, we find that
$$
\theta_Q (\tau) = 1 + q^3 + 2 q^{12} + 6 q^{15} + 6 q^{16} + 6 q^{20} + 6 q^{23} + 2 q^{27} + 6 q^{31} + 6 q^{36} + 6 q^{44} + 12 q^{47} + 8 q^{48} + \cdots 
$$
Proposition \ref{thm:main2} tells us that 
$$
\theta_Q(\tau) -\frac{6}{5} \sum H_{11,11}(n) q^n \in S_{\frac{3}{2}}^+ (4 \cdot 11).
$$
By the work of Shimura \cite{Shimura}, this space is isomorphic under the Shimura lifts to the 1-dimensional space $S_2(11) = \eta(2\tau)^2 \eta(11 \tau)^2 \mathbb{C}$ and we have $ S_{\frac{3}{2}}^+ (4 \cdot 11) = f \mathbb{C}$, where
 $$
f (\tau) = ( \theta (11 \tau) \eta(2\tau) \eta(22 \tau) ) | U_4 = q^3 - q^4 -q^{11} - q^{12}  + q^{15} + 2 q^{16} +q^{20} - q^{23} - q^{27} - q^{31} + q^{44} + q^{55} + \cdots 
 $$
 by [\cite{Bruinier-Kohnen}, $\S 4$ Example 2] and 
     \[
    \Big(\sum_{n=0}^\infty a(n)\, q^n\Big)|U_4\coloneqq\; \sum_{n=0}^\infty a(4n)\, q^n.
    \]
 We obtain the relation
 $$
 \theta_Q(\tau) - \frac{6}{5} \sum H_{11,11}(n) q^n  = \frac{1}{5} f(\tau).
 $$
\end{exa}

Theorems \ref{thm:Main4} and Theorem \ref{thm:Main} are proved using Theorem \ref{thm:PWtoLZ}. The key idea is to express both generalized class numbers in terms of their local factors. We show that these local factors satisfy certain explicit linear relations. As a consequence, we can write $H_{m,N}(n)$ as a linear combination of $H^{(m,\frac{N}{m})}(n)$, with explicit formulas for the coefficients (and vice versa).

\begin{thm}\label{thm:PWtoLZ}
    Let $N>1$ be an odd square-free integer and $m$ be a divisor of $N$. Then, for $n \ge 0$, we have
    $$
    \frac{2^{\omega(\frac{N}{m})} m}{N} H_{m,N}(n) = \sum_{g |\frac{N}{m}} u(g) H^{(mg, \frac{N}{mg})}(n).
    $$ 
    where
    $\omega(N) =\displaystyle\sum_{p \text{ prime}, \ p \mid N} 1$ counts the distinct prime divisors of $N$ and
    \begin{align*}
        u(g) = \left( \displaystyle\prod_{p \mid g} \frac{1}{p-1} \right) \left( \displaystyle\prod_{p \mid \frac{N}{mg}} \frac{1}{p+1} \right).
    \end{align*}
    In particular, we have $H_{N,N}(n)=H^{(N,1)}(n)$.
\end{thm}

Combining Theorem \ref{thm:PWtoLZ} with the known modular properties of the $H_{m,N}(n)$ generating function, we deduce that the modular properties of the $\Hs^{(m,\frac{N}{m})}(\tau)$ functions are as follows:
\begin{cor}\label{thm:harmonic-Maass-form}
Let $N$ be odd and square-free.  When $m \neq 1$, $\Hs^{(m,\frac{N}{m})}(\tau) \in E_{\frac{3}{2}}^+(4N)$. When $m=1$, we have that
    $$
    \frac{1}{2^{\omega(N)}} \Hs^{(1,N)}(\tau) + \frac{1}{8\pi\sqrt{v}} + \frac{1}{4\sqrt{\pi}} \sum n \Gamma\left(-\frac{1}{2}; 4\pi n^2 v\right) q^{-n^2}.
    $$ 
    is a harmonic Maass form (see Section \ref{sec:Modularity} for the definition) of weight $\frac{3}{2}$ with respect to $\Gamma_0(4N)$. 
\end{cor}
\begin{rmk}
    Proposition \ref{thm:harmonic-Maass-form} says that $\Hs^{(1,N)}(\tau)$ is a weight-$3/2$ mock modular form for $\Gamma_0(4N)$ (see Section \ref{sec:corollaries} for the definition).
\end{rmk}

The method used to prove Theorem~\ref{thm:PWtoLZ} is also useful for proving relations among the class numbers $H_{m,N}(n)$. For example, we prove a generalization of Theorem~1.1 of \cite{Beckwith-Mono-1} (see Theorem~\ref{thm:CltoPW} below), expressing the classical Hurwitz class number as a linear combination of generalized ones.

The outline of the paper is as follows. In Section \ref{sec:prelims}, we review the background on modular forms of half-integral weight, ternary quadratic forms, and generalized Hurwitz class numbers. In Section 3, we establish linear relations among local factors and prove Theorem \ref{thm:PWtoLZ}, as well as an extension to the case of non-trivial character $\chi_\ell$ (Theorem \ref{thm:nontrivial}). Section 4 presents some applications, including Theorem \ref{thm:Main}, Proposition \ref{thm:main2}, a description of the mock modular properties of the numbers $H^{(1,N)}(n)$, a generalization of Theorem 1 of \cite{Beckwith-Mono-1}, and Theorem \ref{thm:Main4}.

\section{Preliminaries}\label{sec:prelims}

\subsection{Half-integral weight modular forms}\label{sec:halfintegral}
Let $k \in \frac{1}{2}\Z$, let $N$ be a positive integer and $\chi$ be a Dirichlet character modulo $N$. Choosing the principal branch of the square-root throughout, the \emph{(Petersson) slash operator} is defined as 
\begin{align*} 
\left(f\vert_k\gamma\right)(\tau) \coloneqq \begin{cases}
(c\tau+d)^{-k} f(\gamma\tau) & \text{if } k \in \Z, \\
\left(\frac{c}{d}\right)\varepsilon_d^{2k}(c\tau+d)^{-k} f(\gamma\tau) & \text{if } k \in \frac{1}{2}+\Z,
\end{cases}
\quad \gamma = \left(\begin{matrix} a & b \\ c& d \end{matrix}\right) \in \begin{cases}
\slz & \text{if } k \in \Z, \\
\Gamma_0(4) & k \in \frac{1}{2}+\Z,
\end{cases}
\end{align*}
where $
\varepsilon_d =
\begin{cases}
1 & \text{if } d \equiv 1 \pmod*{4}, \\
i & \text{if } d \equiv -1 \pmod*{4}.
\end{cases}
$
\begin{defn}
Let $k \in \frac{1}{2}\Z$, $N \in \N$, and $f \colon \Hb \to \C$ be a smooth function. Assume that $4 \mid N$ whenever $k \in \frac{1}{2} + \Z$.
We say that $f$ is a \emph{(holomorphic) modular form} of weight $k$ on $\Gamma_0(N)$, if $f$ has the following properties:
\begin{enumerate}[label={\rm (\roman*)}]
\item we have $f\vert_k \gamma = f$ for all $\gamma \in \Gamma_0(N)$, 
\item $f$ is holomorphic on $\Hb$,
\item we have that $f$ is holomorphic at every cusp of $\Gamma_0(N)$.
\end{enumerate}
We let $M_k (\Gamma, \chi)$ be the space of such functions. If $f \in M_k(N,\chi)$ vanishes at all cusps, we say it is a cusp form. The space of cusp forms is denoted $S_k(N,\chi)$, and its orthogonal complement (with respect to the Petersson inner product) is the Eisenstein subspace $E_k (N, \chi)$. 
\end{defn}
A modular form $f \in M_k(N,\chi)$ is an element of \emph{Kohnen's plus space} $M_k^+ (N,\chi)$ if $k \in \frac{1}{2}+\Z$ and its Fourier coefficients are supported on indices $n$ satisfying $(-1)^{k-\frac{1}{2}} n \equiv 0$, $1 \pmod*{4}$. The cusp forms in $M_k^+ (N,\chi)$ form the space $S_k^+ (N,\chi)$ and the Eisenstein series in $M_k^+ (N,\chi)$ comprise the space $E_k^+ (N,\chi)$. 
We typically suppress the character in the notation when it is trivial.

\subsection{The Eisenstein plus-space}\label{sec:Eisenstein}
Let $N > 1$ be odd and square-free and define
\begin{align}\label{eq:L_Ndef}
\chi_{d} &\coloneqq \left(\frac{d}{\cdot}\right), \quad \chi_d' \coloneqq \left( \frac{\cdot }{d} \right),  \nonumber \\
L_{N}(s,\chi) &\coloneqq L(s,\chi)\prod_{\substack{p \text{ prime} \\ p \mid N}} \left(1-\chi(p)p^{-s}\right) = \prod_{\substack{p \text{ prime} \\ p \nmid N}}\frac{1}{1-\chi(p)p^{-s}} = \sum_{\substack{n \geq 1 \\ \gcd(n,N)=1}} \frac{\chi(n)}{n^s}. 
\end{align}
Pei and Wang \cite{pei-wang} (Theorem 1.1 I) showed that the dimension of $E_{\frac{3}{2}}^+(4N)$ equals $2^{\omega(N)}-1$. Let $\ell \mid N$ and define
\begin{align*}
\sigma_{\ell,N,s}(r) \coloneqq
\sum_{\substack{d \mid r \\ \gcd(d,\ell)=1 \\ \gcd\left(\frac{r}{d},\frac{N}{\ell}\right)=1}} d^s, \qquad \sigma_{N,s}(r) \coloneqq \sigma_{N,N,s}(r).
\end{align*} 
For $\ell | N$, $\varepsilon:=(-1)^{{\frac{\ell-1}{2}}}$, the Fourier coefficients in weight $\frac{3}{2}$ are given by
\begin{align} \label{eq:HNNdef}
H(\ell,N,N;n) 
&\coloneqq \begin{cases}
L_N\left(-1,\mathrm{id}\right) & \text{if } n=0, \\
L_N(0,\chi_{D_{\ell,n}'}) \sum\limits_{\substack{a \mid f_n}} \mu(a) \chi_{D_{\ell,n}}(a) \chi_{\ell}'(a) \sigma_{N,1}\left(\frac{f_n}{a}\right) & \begin{array}{@{}l} \text{if } - \varepsilon n=D_{\ell,n}f_n^2 \\-\ell n=D_{\ell,n}'(f_n')^2, \\ D_{\ell,n}, D_{\ell,n}' \text{ fundamental}; \end{array}  \\
0 & \text{else},
\end{cases}
\end{align}
as well as by
\begin{align} \label{eq:HellNdef}
H (\ell, m, N; n)
&\coloneqq \begin{cases}
0 & \text{if } n=0; \\
L_m \left(0,\chi_{D_{\ell,n}'} \right) \prod\limits_{\substack{p \text{ prime} \\ p\mid \frac{N}{m}}} \frac{1-\left(\frac{D_{\ell,n}'}{p}\right)p^{-1}}{1-p^{-2}}\Big(\frac{(\ell,D_{\ell,n})}{(\ell,D_{\ell,n},m)}\Big)\\ 
\qquad \times \sum\limits_{\substack{a \mid f_n }} \mu(a) \chi_{D_{\ell,n}}(a) \chi_{\ell}'(a) \sigma_{m,N,1}\left(\frac{f_n}{a}\right) & \begin{array}{@{}l} \text{if } -\varepsilon n=D_{\ell,n}(f_n)^2, \\ - \ell n = D_{\ell,n}' (f_n')^2,\\ D_{\ell,n}, D_{\ell,n}'\text{ fundamental};
\end{array}  \\ 
0 & \text{else},
\end{cases}
\end{align}
for $N \neq m | N$. 
Define
$$
\Hs_{\ell,m,N}(\tau) := \sum_{n=0}^{\infty} H( \ell, m, N; n) q^n.
$$

\begin{thm}[\cite{pei-wang} Theorem 1]\label{thm:PW}
If $N >1$ is an odd square-free integer, the functions $\{ H_{\ell, m, N} (\tau) : 1 < m  | N \}$ form a basis of $E_{\frac{3}{2}}^+(4N, \chi_{\ell})$.
\end{thm}

\subsection{Ternary Quadratic forms}
Let $Q$ be a positive definite integral ternary quadratic form given by:
\[
Q(x,y,z) = ax^2 + by^2 + cz^2 + ryz + sxz + txy.
\]
\medskip
Recall that the matrix associated to $Q$ is
$$
M_Q = 
\begin{pmatrix}
2a & t & s \\
t & 2b & r \\
s & r & 2c
\end{pmatrix}.
$$
\
The \emph{discriminant} of $Q$ is defined as $d = d_Q := \frac{\det(M_Q)}{2}=4abc+rst-ar^2-bs^2-ct^2$. The \emph{level} of $Q$ is the smallest positive integer $N$ such that 
  $NM_f^{-1}$ is even, that is, has integral entries and even
integers on the main diagonal. The \emph{theta series} of the form $Q$ is defined by: 
  $$
  \theta_Q(\tau):=\sum_{(a,b,c) \in \Z^3} q^{Q(a,b,c)}=\sum_{n=0}^{\infty}R_Q(n)q^n.
  $$
where $$R_Q(n) := \Bigg| \left\{ X \in \mathbb{Z}^3 \mid Q(X) = \frac{1}{2} X^t M_Q X = n \right\}\Bigg|$$ is the number of ways to represent an integer $n$ by ternary quadratic form $Q$.

Let $Q_1, Q_2$ be two ternary quadratic forms. We say that they are equivalent over $\mathbb{Z}$ (or $\mathbb{R}$ or $\mathbb{Z}_p$)
if there exists an invertible real matrix $M \in \mathrm{GL}_3(\mathbb{Z})$ (or $\mathrm{GL}_3(\mathbb{R})$ or $\mathrm{GL}_3(\mathbb{Z}_p)$ respectively) such that
$$
M_{Q_1} = M M_{Q_2} M^t.
$$\
We simply say that $Q_1, Q_2$ are \emph{equivalent} and in the same \emph{class} if they are equivalent over $\mathbb{Z}$, and we say that $Q_1$ and $Q_2$ are \emph{semi-equivalent} and in the same \emph{genus} if they are equivalent over $\mathbb{R}$ and $\mathbb{Z}_p$; that is, they become isometric after extension to each local completion of $\mathbb{Q}$.\

Let $G_{4 N,\,d,\,N^{odd}}$ be the set equivalence classes of positive definite ternary quadratic forms $Q$ of level $N$, discriminant $d$ which are anisotropic  at primes $p$ (i.e. the forms have no non-trivial zero in $\mathbb{Q}_p$) only for $p | N^{odd}$. Theorem 1.1 of \cite{LuoZhou} says that $G_{4N, \, d, N^{odd}}$ is a genus, and moreover, that every genus of level $4 N$ forms is of the form $G_{4 N,\,d,\,N^{odd}}$, where $d, N^{odd}$ are classified in terms of $N$.

\subsection{Ternary theta functions}\label{sec:theta-functions}
Let $D\ge0$ and $N_1, N_2$ be two odd coprime square-free positive integers.
Let $f_{N_1,N_2}$ be the largest integer containing only prime factors of $N_1N_2$ such that $-D/f_{N_1,N_2}^2$ is still a negative discriminant.

Li, Skoruppa, and Zhou \cite{LiSkoruppaZhou} define the \emph{modified class number} as
\begin{align*}
&H^{(N_1,N_2)}(D) 
:= \\ 
& 
\begin{cases} H\left(\frac{-D}{f_{N_1,N_2}^2}\right) 
\displaystyle\prod_{p \mid N_1} \left(1 - \left(\frac{-D/f_{N_1,N_2}^2}{p}\right)\right)   
\prod_{p \mid N_2} 
\frac{2p f_p - p - 1 - \left(\tfrac{-D/f_{N_1,N_2}^2}{p}\right)(2f_p - p - 1)}{p-1} & \text{if } D \neq 0 \\
- \frac{1}{12} \displaystyle\prod_{p|N_1 } (1-p) \prod_{p|N_2 } (p+1) & \text{if } D=0
\end{cases}
\end{align*}
where $f_p$ is the $p-$part of $f_{N_1,N_2}$.

Luo and Zhou \cite{LuoZhou} prove that a certain linear combination of representation of quadratic forms varying over the equivalence classes in genera is given by a modified Hurwitz class number. The case of their result that we will use is the following:
\begin{thm}[Theorem 1.2 \cite{LuoZhou}]\label{thm:L&Z}
Let $N$ be a odd square-free integer and $N^{\mathrm{odd}}$ be a divisor of $N$ with an odd number of prime factors. Then, we have:
\begin{align*}
\sum_{Q \in G_{4N,\,16N^2,\,N^{\mathrm{odd}}}}
 \frac{
\theta_Q(\tau)}{|\mathrm{Aut}(Q)|}
= 2^{-\omega(N)-1} \left(\sum_{n=0}^{\infty}H^{(N^{\mathrm{odd}},\,N/N^{\mathrm{odd}})}(n)q^n\right),
\end{align*}
where $\mathrm{Aut}(Q) = \{ M \in \mathrm{GL}_3(\mathbb{Z}): M_Q = M M_Q M^t \}$.
\end{thm}

\textbf{Remark:} [\cite{LuoZhou}, Theorem 1.2] also give us the equation
$$
\sum_{Q \in G_{4N,\,4N^2,\,2N^{\mathrm{even}}}}
\frac{R_Q(n)}{|\mathrm{Aut}(Q)|}
=
2^{-\omega(N)-2}
H^{(2N^{\mathrm{even}},\,N/N^{\mathrm{even}})}(4n).
$$
where $N^{\mathrm{even}}$ is a divisor of $N$ with an even number of prime factors. In particular, when $N=1$, we have
\begin{equation}\label{eq:r3}
    r_3(n)=12 H^{(2,1)}(4n).
\end{equation}

As one of the applications of their main results, when $N$ is odd and square-free, Theorem 7.2 of \cite{LuoZhou} says that $\{ \theta_{d, 2N/d} (\ell z) : d |2N, d > 1 \}$ is a basis of $E_{\frac{3}{2}}(4N, \chi_{\ell})$, where 
$$
\theta_{d, 2N/d} ( z) := \sum_{n=0}^{\infty} H^{(d, 2N/d)}(4n) q^n. 
$$

\section{Relations between higher level Hurwitz class numbers}
We will express $H(\ell, m, N; n)$ as a linear combination of $H^{(d, N/d)}(n)$ and vice versa when $\ell=1$ (the trivial character case $\chi_\ell=id$).
For simplicity throughout set
$$
H_{m,N}(n) := H(1,m,N;n).
$$
and $D_n:=D_{1,n}$. Note that when $\ell=1$, we have $\varepsilon=(-1)^{\frac{\ell-1}{2}}=1$ and the equation $-\varepsilon n=D_{\ell,n}f_n^2$ becomes $-n = D_n f_n^2$.
\subsection{Local factors} Let $-n = D_n f_n^2$ wehre $D_n$ is a fundamental discriminant and let $p$ be a prime number. Local factors can be defined as follows:
$$
\begin{aligned}
A_p(n) &= \frac{f_p \bigl(1 - \left(\tfrac{D_n}{p}\right) p^{-1}\bigr)}{1 - p^{-2}}, \\
B_p(n) &= \frac{2f_p - p - 1 - \left( \frac{D_n}{p} \right)(2f_p - p - 1)}{p-1}, \\
C_p(n) &= 1 - \left(\frac{D_n}{p}\right), \\
D_p(n) &= \sigma(f_p) - \sigma\!\left(\tfrac{f_p}{p}\right)\!\left(\tfrac{D_n}{p}\right),
\end{aligned}
$$
where $\sigma(1/p)=0$.

The following lemma will be useful later on.
\begin{lemma}\label{lem:linearlocalfactor}
Let $p$ be a prime number, $n \ge 1$. Then
{\normalfont
\begin{enumerate}
    \item $B_p(n) = (p+1) \left( \frac{2}{p} A_p(n) + \frac{1}{1-p} C_p(n) \right)$,
    \item $A_p(n)=\dfrac{p}{2(p+1)}B_p(n)+\dfrac{p}{2(p-1)}C_p(n)$,
    \item $D_p(n) = \frac{p+1}{p} A_p(n) + \frac{1}{1-p} C_p(n)$,
    \item $D_p(n)=\frac{1}{2}B_p(n)+\frac{1}{2}C_p(n).$
\end{enumerate}
}
\end{lemma}
\begin{proof}
First, we expand $B_p(n)$ and compute
\begin{align*}
B_p(n)  &= \frac{2 f_p \left(p - \left(\frac{D_n}{p}\right)\right)}{p-1} - \frac{\left(1 - \left(\frac{D_n}{p}\right)\right) (p+1)}{p-1} \\
    &= \frac{2(p+1)}{p} \frac{f_p p \left(p - \left(\frac{D_n}{p}\right)\right)}{(p-1)(p+1)} - C_p(n) \frac{p+1}{p-1} \\
    &= (p+1) \left( \frac{2}{p} A_p(n) + \frac{1}{1-p} C_p(n) \right)
\end{align*}
The second part is a direct consequence of the first part. To prove the third part, let $f_p = p^T$ for some positive integer $T$. Then writing the divisor sum evaluated at $f_p$ as a geometric sum, we obtain
\begin{align*}
    D_p(n) &= p^T + \left(1 - \left(\frac{D_n}{p}\right)\right) \frac{p^T - 1}{p-1} \\
        &= \frac{p^T\left(p-1 + 1 - \left(\frac{D_n}{p}\right)\right)}{p-1} - \frac{1 - \left(\frac{D_n}{p}\right)}{p-1}\\
        &= \frac{p+1}{p} A_p(n) + \frac{1}{1-p} C_p(n).
\end{align*}
The last part follows from combining the first and third parts of the lemma.
\end{proof}

\begin{lemma}\label{lem:SintoP}
    Let $-n = D_n(f_n)^2$ where $D_n$ is a fundamental discriminant and $N$ be an odd square-free integer, with $m$ be a positive divisor of $N$. Then
    \begin{align*}
    \sum_{\substack{d | f_n \\ (d, N) = 1}} \mu(d) \left( \frac{D_n}{d} \right) \sigma_{m, N, 1} \left( \frac{f_n}{d} \right) = \prod_{p | \frac{N}{m}} f_p \cdot \prod_{p \nmid N} D_p(n)
\end{align*}
where $f_p$ is the $p$-part of $f_n$.
\end{lemma}
\begin{proof}
By the multiplicativity of $\sigma_{m,N,1}$ and M\"{o}bius inversion, we have
\begin{equation}\label{eq:Lemma32-1}
  \sum_{\substack{d | f_n \\ (d, N) = 1}} \mu(d) \left( \frac{D_n}{d} \right) \sigma_{m, N, 1} \left( \frac{\widetilde{f_n}}{d} \right) = \prod_{\substack{p|f_n \\ p \nmid N}} \left( \sigma_{m,N,1}(f_p) - \left( \frac{D_n}{p} \right) \sigma_{m,N,1}(f_p/p) \right) ,
\end{equation}
    where $\widetilde{f_n} = \prod_{p \nmid N} f_p$. 
Note that if $r$ is a $p$-power where $p$ is a prime such that $p|m$, then one has $\sigma_{m,N,1}(r) = 1$. If on the other hand $r | \frac{N}{m}$, we have $\sigma_{m,N,1}(r) = r$. We deduce that
for any divisor $d$ of $f_n$ coprime to $N$, we have 
$$
\sigma_{m,N,1}(f_n/d) = \prod_{p| \frac{N}{m}} f_p \cdot \sigma_{m,N,1}(\widetilde{f_n}/d).
$$
    Plugging this into \eqref{eq:Lemma32-1} completes the proof.
\end{proof}

\begin{prop}\label{thm:Localexp}
With the above notation, for all $-n=D_nf_n^2< 0$ where $D_n$ is a fundamental discriminant), we have 
      \begin{align} \label{eq:PW}
      H_{m,N}(n) &=  L(0,\chi_{D_n})\displaystyle\prod_{p \nmid N} D_p(n) \displaystyle\prod_{p \mid m} C_p(n) \displaystyle\prod_{p \mid \frac{N}{m}} A_p(n)  
      \end{align}
      \begin{align}\label{eq:LSZ}
          H^{(m, \frac{N}{m})}(n)&= L(0, \chi_{D_n}) \displaystyle\prod_{p \nmid N} D_p(n) \displaystyle\prod_{p \mid m} C_p(n) \displaystyle\prod_{p \mid \frac{N}{m}} B_p(n).
       \end{align}
    In particular, 
    \begin{align}\label{eq:cl}
        H(n) = H_{1,1}(n) = L(0, \chi_{D_n}) \displaystyle\prod_{p \mid f_n} D_p(n) .
    \end{align} 
\end{prop}

\begin{proof}
    We start by proving the first equation. Using Lemma \ref{lem:SintoP}, we compute
    \begin{align*}
    &H_{m,N}(n)\\ 
&= L_m(0, \chi_{D_n}) \prod_{p \mid \frac{N}{m}} \frac{1 - p^{-1} \left(\frac{D_n}{p}\right)}{1 - p^{-2}} \sum_{\substack{d | f_n \\ (d, N) = 1}} \mu(d) \left( \frac{D_n}{d} \right) \sigma_{m, N, 1} \left( \frac{f_n}{d} \right) \\
&= L(0, \chi_{D_n}) 
      \displaystyle\prod_{p \mid m} 
      \left(1 - \left(\frac{D_n}{p}\right)\right) \left(
        \displaystyle\prod_{p \mid \frac{N}{m}} 
        \frac{f_p \left(1 - p^{-1} \left(\frac{D_n}{p}\right)\right)}{1 - p^{-2}}
      \right)
      \
      \frac{1}{\displaystyle\prod_{p \mid \frac{N}{m}} f_p}   \left(\displaystyle\prod_{p \mid \frac{N}{m}} f_p\right)\left(\displaystyle\prod_{p \nmid N} D_p(n)\right) \\
      &= L(0, \chi_{D_n}) 
      \prod_{p \mid m} C_p(n) \cdot 
      \displaystyle\prod_{p \mid \frac{N}{m}} A_p(n)  \cdot  
    \displaystyle\prod_{p \nmid N} D_p(n)
    \end{align*}
This proves \ref{eq:PW}. Equation \ref{eq:cl} is a special case when $n=M=1$.

For the second equation, as in subsection \ref{sec:theta-functions}, let $f_{m, \frac{N}{m}}$ be the largest integer containing only prime factors of $m\cdot\frac{N}{m}=N$ such that $-n/f_{m,\frac{N}{m}}^2$ is still a negative discriminant. Then $f_{m, \frac{N}{m}} = \displaystyle\prod_{\substack{p \mid N}} f_p$
where $f_p$ is the $p$-part of $f_n$ and we have
\[
-\frac{n}{f_{m, \frac{N}{m}}^2}
= D_n \left( \frac{f_n}{\displaystyle\prod_{p \mid N} f_p} \right)^2
= D_n \big(\displaystyle\prod_{p \nmid N} f_p\big)^2.
\]
Applying equation \ref{eq:cl} for $n=\frac{n}{f_{m, \frac{N}{m}}^2}$, we find that
\[
H\!\left(\frac{n}{f_{m, \frac{N}{m}}^2}\right)
= L(0, \chi_{D_n}) \displaystyle\prod_{p \mid (f/\prod_{p|N} f_p)} D_p\left( \frac{n}{f_{m, \frac{N}{m}}^2} \right)
= L(0, \chi_{D_n}) \displaystyle\prod_{p \nmid N} D_p(n).
\]
Now, we have
\begin{align*}
H^{(m, \frac{N}{m})}(n)
&= H\left(-\frac{n}{f_{m, \frac{N}{m}}^2}\right)
   \displaystyle\prod_{p \mid m}
   \left(1 - \left(\frac{(-n / f_{m, \frac{N}{m}}^2)}{p} \right)\right) \\
   &\hspace{.3in} \cdot
   \displaystyle\prod_{p \mid \frac{N}{m}}
   \frac{2pf_p - p - 1 - \left(\frac{-n / f_{m, \frac{N}{m}}^2}{p}\right)(2f_p - p - 1)}{p - 1} \\[6pt]
&= L(0, \chi_{D_n})
   \displaystyle\prod_{p \nmid N} D_p(n)
   \displaystyle\prod_{p \mid m} C_p(n)
   \displaystyle\prod_{p \mid \frac{N}{m}} B_p(n).
\end{align*}
\end{proof}

\subsection{Proof of Theorem \ref{thm:PWtoLZ}}
Using Proposition \ref{thm:Localexp}, we can relate the two types of generalized Hurwitz class numbers. First we express Pei and Wang's class numbers in terms of Li--Skoruppa--Zhou class numbers, which is Theorem \ref{thm:PWtoLZ}. 
\begin{proof}[Proof of Theorem \ref{thm:PWtoLZ}]
   First, let us check the $n=0$ case. When $m=N$, the left hand side is $L_N(-1,\mathrm{id})$ and the right hand side is $H^{(N,1)}(0) = - \frac{1}{2} \prod_{p|N} (1-p) = H(-1,id)$.
   When $m < N$, the left-hand side is 0 and the right-hand side is 
   \begin{align*}
       \sum_{g|(\frac{N}{m})} u(g)  H(0) \prod_{p|\frac{N}{mg}} \frac{-p-1}{p-1} &= H(0) \cdot  \big(\prod_{p|(\frac{N}{m})} \frac{1}{p-1}\big) \cdot\sum_{g|(\frac{N}{m})} \mu \left(\frac{N}{mg} \right) =0.
    \end{align*}
    where in the last step we use $\displaystyle\sum_{g|(\frac{N}{m})} \mu(g)=0$ when $m<N$.
   
    For $n>0$, by Theorem \ref{thm:Localexp} and Lemma \ref{lem:linearlocalfactor} (part $2$), we have the following:
    \begin{align*}
    H_{m,N}(n) &= L(0, \chi_{D}) \displaystyle\prod_{p \nmid N} D_p(n) \displaystyle\prod_{p \mid m} C_p(n) \displaystyle\prod_{p \mid \frac{N}{m}} A_p(n)\\
    &=L(0, \chi_{D}) \displaystyle\prod_{p \nmid N} D_p(n) \displaystyle\prod_{p \mid m} C_p(n) \displaystyle\prod_{p \mid \frac{N}{m}}\left(\dfrac{p}{2(p+1)}B_p(n)+\dfrac{p}{2(p-1)}C_p(n)\right)
    \end{align*}
    This implies that:
    \begin{align*}
    \frac{H_{m,N}(n)}{\left( \displaystyle\prod_{p \mid \frac{N}{m}} p/2 \right)}  &=L(0, \chi_{D}) \displaystyle\prod_{p \nmid N} D_p(n) \displaystyle\prod_{p \mid m} C_p(n) \displaystyle\prod_{p \mid \frac{N}{m}}\left(\dfrac{1}{(p+1)}B_p(n)+\dfrac{1}{(p-1)}C_p(n)\right) \\
    &= L(0, \chi_{D}) \displaystyle\prod_{p \nmid N} D_p(n) \prod_{p \mid m} C_p(n)  \sum_{g | \frac{N}{m}}   \prod_{p \mid g} \frac{C_p(n) }{p-1} \prod_{p \mid \frac{N}{mg}} \frac{B_p(n)}{p+1}\\
     &= L(0, \chi_{D}) \displaystyle\prod_{p \nmid N} D_p(n)   \sum_{g | \frac{N}{m}} u(g)  \prod_{p \mid m g}  C_p(n) \prod_{p \mid \frac{N}{mg}}B_p(n)\\
    &=\sum_{g | \frac{N}{m}} u(g) \cdot H^{(mg, \frac{N}{mg})}(n). 
\end{align*}
\end{proof}

\subsection{Connecting various class numbers}
Next, we express Luo and Zhuo's class numbers in terms of Pei and Wang's class numbers.
\begin{thm}\label{thm:LZtoPW}
For any odd square-free $N$, we have
        $$
        \frac{1}{\displaystyle\prod_{p \mid \frac{N}{m}} (p+1)} H^{(m, \frac{N}{m})} (n)
        = \sum_{v | \frac{N}{m}} t(v) \cdot  H_{mv,N}(n) ,
        $$
        where
        $$
        t(v) = \left( \prod_{p \mid \tfrac{N}{mv}} \frac{2}{p} \right) 
        \left( \prod_{p \mid v} \frac{1}{1-p} \right).
        $$
\end{thm}
\begin{proof}
    The proof is analogous to that of Theorem \ref{thm:PWtoLZ}, using the first claim of Lemma \ref{lem:linearlocalfactor} instead of the second. 
\end{proof}

We can express $H^{(m,\frac{N}{m})}(n)$ in terms of Pei and Wang's class numbers for nontrivial characters  $\chi_{\ell}$ (i.e. $l\neq 1$) as follows:
\begin{thm}\label{thm:nontrivial}
   Let $\ell \neq 1$ be a divisor of $N$, where as before, $N$ is an odd square-free integer and $m | N$.  Then we have:
        $$
        \frac{1}{\displaystyle\prod_{p \mid \frac{N}{m}} (p+1)} H^{(m,\frac{N}{m})}(n)= \sum_{v \mid \frac{N}{m}} \frac{t(v)}{(\ell, \frac{N}{m}v)} H(\ell, mv, N; \ell n).
        $$
\end{thm}
\begin{proof}
    This follows from combining Theorem \ref{thm:LZtoPW} and Proposition \ref{prop:Ape3}.
\end{proof}

\begin{prop}\label{prop:Ape3} Let $1<\ell\mid N$ and $m\mid N$. We have:
$$
H(\ell,m,N;\ell n)=\left(\ell,\frac{N}{m}\right)H_{m,N}(n).
$$
\end{prop}
\begin{proof}
First we generalize our local factors. Let $1<\ell$ be a divisor of $N$, $-\varepsilon n=D_{\ell,n}(f_n)^2$ and $-\ell n=D_{\ell,n}'(f_n')^2$ (as in the definition of $H(\ell,m,N;n)$). Let
\begin{align*}
    A_p(\ell, n) &= \frac{f_p\left(1 - \left(\frac{D_{\ell,n}'}{p}\right) p^{-1}\right)}{1 - p^{-2}} = \frac{p f_p \left(p - \left(\frac{D_{\ell,n}'}{p}\right)\right)}{(p-1)(p+1)} \\
    C_p(\ell, n) &= 1 - \left(\frac{D_{\ell,n}'}{p}\right) \\
    D_p(\ell, n) &:= \sigma_1(f_p) - \sigma_1\left(\frac{f_p}{p}\right) \left(\frac{D_{\ell,n}}{p}\right) \chi_\ell'(p)
\end{align*}
where $f_p$ is the $p$-part of $f_n$. Note that in case $\ell=1$, we have $D_{\ell,n}=D_{\ell,n}'=D_n$, and hence $A_p(\ell,n)=A_p(n)$, $ C_p(\ell,n)=C_p(n)$ and $D_p(\ell,n)=D_p(n).$
We can extend the proof of Lemma \ref{lem:SintoP} to show the relation
     $$\displaystyle\sum_{\substack{d | f_n \\ (d, N) = 1}} \mu(d) \chi'_l(d) \left( \frac{D_{\ell,n}}{d} \right) \sigma_{m, N, 1} \left( \frac{f_n}{d} \right)
    = \displaystyle \prod_{p | \frac{N}{m}} f_p  \cdot \prod_{p \nmid N} D_p(\ell,n)$$
    This enables us to extend the proof of Proposition \ref{thm:Localexp} to show
        $$H(\ell, m, N; n) = L(0, \chi_{D_{\ell,n}'}) \displaystyle\prod_{p \nmid N} D_p(\ell, n) \displaystyle\prod_{p \mid m} C_p(\ell, n)
        \displaystyle\prod_{p \mid \frac{N}{m}} A_p(\ell, n) \cdot \left( \frac{(\ell, D_{\ell,n})}{(\ell, D_{\ell,n}, m)} \right)$$
     
    Finally, we substitute $\ell n$ for $n$ in the previous equation and use the following relations to connect $H(\ell,m,N;\ell n)$ with $H_{m,N}(n)$:  
    \begin{enumerate}
    \item $D'_{\ell,\ell n} = D_n, \quad D_{\ell,\ell n} = \varepsilon \frac{\ell D_n}{(l, D_n)^2}, \quad f_{\ell n} = (\ell, D_n) f_n$
    
    \item $C_p(\ell, \ell n) = C_p(n)$
    
    \item $\displaystyle\prod_{p | \frac{N}{m}} A_p(\ell, \ell n) = (\ell, D_n, \frac{N}{m}) \displaystyle\prod_{p | \frac{N}{m}} A_p(n)$
    
    \item $\frac{(\ell, D_{\ell,\ell n})}{(\ell, D_{\ell,\ell n}, m)} = (\ell, \frac{N}{m}) (\ell, D_{n}, \frac{N}{m})^{-1}$
    
    \item For $p \nmid N, \quad D_p(\ell, \ell n) = D_p(n)$.
\end{enumerate}

\end{proof}
\section{Corollaries of Theorems \ref{thm:PWtoLZ} and \ref{thm:LZtoPW}}\label{sec:corollaries}
We give several consequences of Theorems \ref{thm:PWtoLZ} and \ref{thm:LZtoPW}, which make use of the connections between these generalized Hurwitz class numbers and ternary quadratic forms, the Eisenstein-plus space, and harmonic Maass forms. We also show how the expressions used in the proofs of Theorems \ref{thm:PWtoLZ} and \ref{thm:LZtoPW} can be used to obtain simple identities for the $H_{m,N}(n)$ values, generalizing a theorem in \cite{Beckwith-Mono-1}.

\subsection{Proof of Theorem \ref{thm:Main}} \label{sec:Maintm}
First, we will need an analog of Theorem \ref{thm:L&Z} for divisors $N^{\mathrm{even}}$ of $N$ with an even number of prime factors.

\begin{prop}\label{prop:L&Z}
    For any odd square-free $N$ and any divisor $N^{\mathrm{even}}$ of $N$ with an even number of prime factors, and a prime divisor $q$ of $N^{\mathrm{even}}$, we have:
\begin{align*}
    &H^{(N^{\mathrm{even}},\, N/N^{\mathrm{even}})}(n)\\
&= 2^{\omega(N)+2}
\sum_{Q \in G_{4N/q,\,16N^2/q^2,\,N^{\mathrm{even}}/q}}
\frac{R_Q(n)}{|\mathrm{Aut}(Q)|}
- 2^{\omega(N)+1}
\sum_{Q \in G_{4N,\,16N^2/q^2,\,N^{\mathrm{even}}/q}}
\frac{R_Q(n)}{|\mathrm{Aut}(Q)|}.
\end{align*}
\end{prop}
\begin{proof}
    By Lemma \ref{lem:L&Z} and Theorem \ref{thm:L&Z},
\begin{align*}
&H^{(N^{\mathrm{even}},\, N/N^{\mathrm{even}})}(n)
= 2\,
   H^{(N^{\mathrm{even}}/q,\, N/N^{\mathrm{even}})}(n)
   - H^{(N^{\mathrm{even}}/q,\, Nq/N^{\mathrm{even}})}(n) \\
&= 2^{\omega(N)+2}
   \sum_{Q \in G_{4N/q,\,16N^2/q^2,\,N^{\mathrm{even}}/q}}
   \frac{R_Q(n)}{|\mathrm{Aut}(Q)|}
   - 2^{\omega(N)+1}
   \sum_{Q \in G_{4N,\,16N^2/q^2,\,N^{\mathrm{even}}/q}}
   \frac{R_Q(n)}{|\mathrm{Aut}(Q)|}.
\end{align*}
\end{proof}

\begin{lemma}\label{lem:L&Z} Let $N_1>1$ and $N_2$ be two odd coprime square-free integers, $q$ be a prime factor of $N_1$. Then we have
$$
H^{(N_1,N_2)}(n)=2H^{(N_1/q,N_2)}(n)-H^{(N_1/q,N_2q)}(n).
$$
\end{lemma}
\begin{proof}
    Let $N = N_1 N_2$ and $-n=D_nf_n^2$ where $D_n$ is a fundamental discriminant. By Proposition \ref{thm:Localexp}, we have
\begin{align*}
H^{(N_1,N_2)}(n)
&= L(0,\chi_{D_n})
   \prod_{\substack{p \mid f_n \\ p \nmid N}} D_p(n)
   \prod_{p \mid N_2} B_p(n)
   \prod_{p \mid N_1} C_p(n) .
\end{align*}
Similarly,
\begin{align*}
H^{(N_1/q,N_2)}(n)
&= L(0,\chi_{D_n})
   \prod_{\substack{p \mid f_n \\ p \nmid N/q}} D_p(n)
   \prod_{p \mid N_2} B_p(n)
   \prod_{p \mid N_1/q} C_p(n) ,
\end{align*}
and
\begin{align*}
H^{(N_1/q,N_2q)}(n)
&= L(0,\chi_{D_n})
   \prod_{\substack{p \mid f_n \\ p \nmid N}} D_p(n)
   \prod_{p \mid N_2q} B_p(n)
   \prod_{p \mid N_1/q} C_p(n) .
\end{align*}
Therefore,
$$
\frac{2H^{(N_1/q,N_2)}(n)-H^{(N_1/q,N_2q)}(n)}{H^{(N_1,N_2)}(n)}
=2\frac{D_q(n)}{C_q(n)}-\frac{B_q(n)}{C_q(n)}=\frac{2D_q(n)-B_q(n)}{C_q(n)}=1,
$$
where the last equality follows from the last part of Lemma \ref{lem:linearlocalfactor}
\end{proof}

We are now ready to prove Theorem \ref{thm:Main}.
\begin{proof}[Proof of Theorem \ref{thm:Main}]
    Theorem \ref{thm:PWtoLZ} implies that
    \begin{align*}
    \frac{2^{\omega(\frac{N}{m})} m}{N} \Hs_{m,N}(\tau) = \sum_{g |\frac{N}{m}} u(g)\Big(\sum_{n\ge 0} H^{(mg, \frac{N}{mg})}(n)q^n\Big).
    \end{align*} 
    Using Theorem \ref{thm:L&Z} for the terms $g$ with $\mu(mg)=-1$ and Proposition \ref{prop:L&Z} when $\mu(mg)=+1$, we can show that $\displaystyle\sum_{n\ge 0} H^{(mg, \frac{N}{mg})}(n)q^n=2^{\omega(N)+1}T_{mg,N}(\tau)$.
\end{proof}

\begin{exa}
    Consider $m=N=11$, Theorem \ref{thm:Main} implies
    \[
    \Hs_{11,11}(\tau)=\displaystyle\sum_{Q \in G_{44,16*121,11}}\frac{\theta_Q(\tau)}{|\mathrm{Aut}(Q)|}=\frac{\theta_{Q_1}(\tau)}{|\mathrm{Aut}(Q_1)|}+\frac{\theta_{Q_2}(\tau)}{|\mathrm{Aut}(Q_2)|}
    \]
    where $Q_1(x,y,z)=3x^2-2xy-2xz+15y^2+15z^2-14yz$ and $Q_2(x,y,z)=4x_1^2+11x_2^2+12x_3^2+4x_1x_3$.
    Equivalently,
    \[
    \frac{1}{|\mathrm{Aut}(Q_1)|}\big(\theta_{Q_1}(\tau)-\frac{6}{5}\sum \Hs_{11,11}(\tau)\big)+\frac{1}{|\mathrm{Aut}(Q_2)|}\big(\theta_{Q_2}(\tau)-\frac{6}{5}\sum \Hs_{11,11}(\tau)\big)=0.
    \]
    In other words, Theorem \ref{thm:Main} tells us that the two cusp forms for Proposition \ref{thm:main2}, in this case, are linearly dependent.
\end{exa}
\begin{exa}
    Consider $m=N=35$. Theorem \ref{thm:Main} tells us that we have $2$ different ways to express $\Hs_{35,35}(\tau)$ (depending on the choice of $q$) as:
    \begin{align*}
        \Hs_{35,35}(\tau)&=\displaystyle\sum_{Q \in G_{4*5,16*25,5}}\frac{\theta_Q(\tau)}{|\mathrm{Aut}(Q)|}-\displaystyle\sum_{Q \in G_{4*35,16*1225,5}}\frac{\theta_Q(\tau)}{|\mathrm{Aut}(Q)|}\\
        &=\Big(\frac{\theta_{Q_1}}{|Aut(Q_1)|}\Big)-\Big(\frac{\theta_{Q_2}}{|Aut(Q_2)|}+\frac{\theta_{Q_3}}{|Aut(Q_3)|}+\frac{\theta_{Q_4}}{|Aut(Q_4)|}\Big)
    \end{align*}
    and
    \begin{align*}
        \Hs_{35,35}(\tau)&=\displaystyle\sum_{Q \in G_{4*7,16*49,7}}\frac{\theta_Q(\tau)}{|\mathrm{Aut}(Q)|}-\displaystyle\sum_{Q \in G_{4*35,16*1225,7}}\frac{\theta_Q(\tau)}{|\mathrm{Aut}(Q)|}\\
        &=\Big(\frac{\theta_{Q_5}}{|Aut(Q_5)|}\Big)-\Big(\frac{\theta_{Q_6}}{|Aut(Q_6)|}+\frac{\theta_{Q_7}}{|Aut(Q_7)|}\Big)
    \end{align*}
    where
    \begin{align*}
        &Q_1(x,y,z)(=Q_5(x,y,z)\text{ in the introduction})=3x^2+7y^2+7z^2-6yz+2xy+2yz \in G_{20,400,5},\\
        &Q_2(x,y,z)=7x^2+20y^2+40z^2+20yz \in G_{140,19600,5},\\
        &Q_3(x,y,z)=3x^2+47y^2+47z^2-46yz-2xz-2xy \in G_{140,19600,5},\\
        &Q_4(x,y,z)=12x^2+12y^2+35z^2-4xy\in G_{140,19600,5},\\
        &Q_5(x,y,z)(=Q_7(x,y,z)\text{ in the introduction})=4x^2+7y^2+8z^2-4xz\in G_{28,748,7},\\
        &Q_6(x,y,z)=4x^2+35y^2+36z^2-4xz \in G_{140,19600,7},\\
        &Q_7(x,y,z)=11x^2+15y^2+39z^2-10yz-6xz-10xy \in G_{140,19600,7}.
    \end{align*}
    The above ternary quadratic forms come from the examples in \cite{LuoZhou}. In particular, we have the following equation among theta series of ternary quadratic forms:
    \begin{align*}
        &\Big(\frac{\theta_{Q_1}}{|Aut(Q_1)|}\Big)-\Big(\frac{\theta_{Q_2}}{|Aut(Q_2)|}+\frac{\theta_{Q_3}}{|Aut(Q_3)|}+\frac{\theta_{Q_4}}{|Aut(Q_4)|}\Big)\\
        &=\Big(\frac{\theta_{Q_5}}{|Aut(Q_5)|}\Big)-\Big(\frac{\theta_{Q_6}}{|Aut(Q_6)|}+\frac{\theta_{Q_7}}{|Aut(Q_7)|}\Big).\
    \end{align*}    
    This equation admits the following interpretation. By Theorem~\ref{thm:L&Z}, the theta series attached to ternary quadratic forms within a fixed genus are governed by the generating functions of generalized Hurwitz class numbers. Consequently, relations among theta series arising from different genera are induced by corresponding relations among these generalized Hurwitz class numbers, of which there are many.
\end{exa}

\subsection{Structure of the Eisenstein plus space and proof of Proposition \ref{thm:main2}}

\begin{proof}[Proof of Proposition \ref{thm:main2}]
The first claim of this proposition follows directly from Theorems~\ref{thm:PW} and \ref{thm:Main}. For the second claim, let $Q$ be any quadratic form on the left hand side Theorem \ref{thm:L&Z} with $N =p$ and $N^{odd} = p$. 
By Theorem \ref{thm:L&Z} and Theorem \ref{thm:PWtoLZ}, we have
$$
T_{p,p}(\tau)
= \frac{1}{2^{\omega(N)+1}} \sum_{n \ge 0} H^{(p,1)}(n)\, q^n
= \frac{1}{2^{\omega(N)+1}} \Hs_{p,p}(\tau).
$$
Moreover, the coefficients of $\theta_Q(\tau)$ in Theorem \ref{thm:L&Z} are all positive. Since $\Hs_{p,p}(\tau)$ lies in the plus space, it follows that $\theta_Q(\tau)$ also lies in the plus space. Consequently, up to a cusp form, $\theta_Q(\tau)$ is a scalar multiple of $\Hs_{p,p}(\tau)$.
\end{proof}

\begin{rmk}
    First, let us recall that we have defined $\Hs_{\ell,m,N}(\tau) = \sum_{n=0}^{\infty} H( \ell, m, N; n) q^n$. For $m \ge1$, $m \mid N$, we define
$$
\Hs^{(m, \frac{N}{m})}(\tau) := \sum_{n=0}^{\infty} H^{(m, \frac{N}{m})}(n) q^n.
$$
    Then the first claim of Proposition \ref{thm:main2} is equivalent with the statement that the set \allowbreak $\left\{ \Hs^{(m, \frac{N}{m})}(\tau)\right\}_{\substack{m \mid N \\ m \neq 1}}$ is a basis of $E_{3/2}^+(4N)$ (from the proof of Theorem \ref{thm:Main}). More specifically, for any $m |N$,
$$
\operatorname{span} \{ \Hs_{1,d,N}(\tau): m|d|N \} = \operatorname{Span} \{ \Hs^{(d,N/d)}(\tau): m|d|N \} .
$$ 
This means that the change-of-basis matrix from $\{\Hs^{(m, N/m)}(\tau)\}_{\substack{m \mid N \\ m \neq 1}}$ to $\{\Hs_{1,m,N}(\tau)\}_{\substack{m \mid N \\ m \neq 1}}$ is upper triangular when the former basis is ordered by increasing number of divisors of $m$, and the latter is ordered accordingly.
\end{rmk}

\subsection{Modular properties of $\Hs^{(m,\frac{N}{m})}(\tau)$}\label{sec:Modularity}

Using Theorem \ref{thm:PWtoLZ}, we can describe the modular properties of the $\Hs^{(m,\frac{N}{m})}(\tau)$ functions. For the case of $\Hs^{(1,N)}(\tau)$, we need to define harmonic Maass forms.

The \emph{weight $k$ hyperbolic Laplace operator} is defined as
\begin{align} \label{eq:Deltadef}
\Delta_k \coloneqq  -v^2\left(\frac{\partial^2}{\partial u^2}+\frac{\partial^2}{\partial v^2}\right) + ikv\left(\frac{\partial}{\partial u} + i\frac{\partial}{\partial v}\right).
\end{align}

We say that a smooth function $f:\mathbb{H} \to \mathbb{C}$ is a \emph{harmonic Maass form of weight $k$} on $\Gamma_0(N)$, if $f$ has the following properties:
\begin{enumerate}[label={\rm  (\roman*)}]
\item we have $f\vert_k \gamma = f$ for all $\gamma \in \Gamma_0(N)$, 
\item we have $\Delta_k f = 0$ for all $\tau \in \Hb$,
\item we have that $f$ is of at most linear exponential growth towards all cusps of $\Gamma_0(N)$.
\end{enumerate}

\emph{Proof of Corollary \ref{thm:harmonic-Maass-form}:}
For the first claim, observe that for $m > 1$, the right-hand side of Theorem \ref{thm:PWtoLZ} expresses $\Hs^{(m,\frac{N}{m})}(\tau)$ in terms of functions $H_{1,d,N} (\tau)$ with $d>1$. Recall that \cite{pei-wang} shows that $H_{1,d,N} (\tau)$ belongs to $E_{\frac{3}{2}}^+(4N)$ whenever $d>1$. So, we have $\Hs^{(m,\frac{N}{m})}(\tau) \in E_{\frac{3}{2}}^+(4N).$

Theorem 1.2 of \cite{Beckwith-Mono-1} asserts that the function
\begin{align*}
\frac{1}{N}\sum_{n \geq 1} H_{1,N}(n)q^n + \Big(\prod_{\substack{p \text{ prime} \\ p \mid N}} \frac{1}{p+1}\Big) \Bigg(\frac{1}{8\pi\sqrt{v}} + \frac{1}{4\sqrt{\pi}} \sum_{n \geq 1} n \Gamma\left(-\frac{1}{2},4\pi n^2 v\right) q^{-n^2}\Bigg)
\end{align*}
is a weight $\frac{3}{2}$ harmonic Maass form on $\Gamma_0(4N)$,
which generalized Zagier's \cite{Zagier1} weight $\frac{3}{2}$ non-holomorphic modular completion of $\Hs_{1,1}$ on $\Gamma_0(4).$ Proposition \ref{thm:main2} implies
$$
\frac{1}{N}H_{1,N}(n) = 2^{-\omega(N)} \sum_{g |N} u(g) H^{(g, N/g)} (n). 
$$
which gives us
\begin{align*}
&\frac{1}{N}\sum_{n \geq 1} H_{1,N}(n)q^n + \Big(\prod_{\substack{p \text{ prime} \\ p \mid N}} \frac{1}{p+1}\Big) \Bigg(\frac{1}{8\pi\sqrt{v}} + \frac{1}{4\sqrt{\pi}} \sum_{n \geq 1} n \Gamma\left(-\frac{1}{2},4\pi n^2 v\right) q^{-n^2}\Bigg) \\
&= 2^{-\omega(N)} \sum_{g |N} u(g)\cdot \Big(\sum _{n\ge 0}H^{(g, N/g)} (n)q^n\Big) + \Big(\prod_{\substack{p \text{ prime} \\ p \mid N}} \frac{1}{p+1}\Big) \Bigg(\frac{1}{8\pi\sqrt{v}} + \frac{1}{4\sqrt{\pi}} \sum_{n \geq 1} n \Gamma\left(-\frac{1}{2},4\pi n^2 v\right) q^{-n^2}\Bigg)
\end{align*}
Since the left hand side as well as generating functions for $H^{(g, N/g)} (n)$ (i.e. $\Hs^{(m,\frac{N}{m})}(\tau)$) for $g>1$ are modular with weight $\frac{3}{2}$ on $\Gamma_0(4N)$, it follows that
\begin{align*}
&\sum_{n \geq 1}  2^{-w(N)} u(1) H^{(1, N)} (n) q^n + \Big(\prod_{\substack{p \text{ prime} \\ p \mid N}} \frac{1}{p+1}\Big) \Bigg(\frac{1}{8\pi\sqrt{v}} + \frac{1}{4\sqrt{\pi}} \sum_{n \geq 1} n \Gamma\left(-\frac{1}{2},4\pi n^2 v\right) q^{-n^2}\Bigg)\\
&=\Big(\prod_{\substack{p \text{ prime} \\ p \mid N}} \frac{1}{p+1}\Big)\Bigg( 2^{-w(N)}\sum_{n \geq 1}  H^{(1, N)} (n) q^n + \frac{1}{8\pi\sqrt{v}} + \frac{1}{4\sqrt{\pi}} \sum_{n \geq 1} n \Gamma\left(-\frac{1}{2},4\pi n^2 v\right) q^{-n^2}\Bigg)
\end{align*}
is weight $\frac{3}{2}$ modular for $\Gamma_0(4N)$ as well. From the Fourier expansion we deduce that this function is a weight $\frac{3}{2}$ harmonic Maass form on $\Gamma_0(4N)$.

\qed
\subsection{Relations among the Pei-Wang class numbers}
    \begin{thm}[Theorem 1.1 \cite{Beckwith-Mono-1}]\label{thm:BM1.1}
        If $N$ is a prime number, then
    \begin{align*}
        H(n)=\dfrac{N+1}{N}H_{1,N}(n)+\dfrac{1}{1-N}H_{N,N}(n)
    \end{align*}
    \end{thm}
    
   Using the expressions for $H_{m,N}(n)$ above, we have the generalization
    \begin{thm}\label{thm:CltoPW}
        Let $N$ be an odd square-free integer. Then we have
        \begin{align*}
        H(n) &= \sum_{m|N} v(m) \cdot H_{m, N}(n) \\
        \text{where} \quad v(m) &= \left( \prod_{p \mid \frac{N}{m}} \frac{p+1}{p} \right) \left( \prod_{p \mid m} \frac{1}{1-p} \right).
        \end{align*}
    \end{thm}
    \begin{proof}
        The proof is analogous to that of Theorem \ref{thm:PWtoLZ}, but uses the third claim of Lemma \ref{lem:linearlocalfactor}.
    \end{proof}

We also have a counterpart for the previous theorem in terms of the $H^{(m,\frac{N}{m})}(n)$ class numbers:
\begin{thm}
        Let $N$ be an odd square-free integer. Then we have 
        \begin{align*}
        H(n) &= \frac{1}{2^{\omega(N)}}\sum_{m|N}H^{(m, \frac{N}{m})}(n).
        \end{align*}
    \end{thm}
    \begin{proof}
        The proof is analogous to that of Theorem \ref{thm:PWtoLZ}, but uses the last claim of Lemma \ref{lem:linearlocalfactor}.
    \end{proof}

\subsection{Generalizations of Gauss' formula \ref{eq:threesquares}}.

Using equation \eqref{eq:r3}, we can express Gauss' formula \ref{eq:threesquares} in terms of generalized Hurwitz class numbers as follows:
$$
H^{(1,1)}(4n)-2H^{(1,1)}(n)=H^{(2,1)}(4n)
$$
Therefore, we can generalize Gauss' formula \eqref{eq:G1} as the following Proposition:
\begin{prop}\label{prop:G1}
   Let $N_1,N_2$ be two odd coprime square-free integers and $p$ be a prime number with $(p,N_1N_2)=1$. We have
   \begin{equation}\label{eq:G1}
       H^{(N_1,N_2)}(p^2 n)\;-\;p\,H^{(N_1,N_2)}(n)
=
H^{(N_1\cdot p,N_2)}(p^2 n).
   \end{equation}
\end{prop}
\begin{proof}
\textbf{Case 1:} Assume $-n$ is not a discriminant (i.e. $n \equiv 1,2 \mod 4$).

If $p \neq 2$, then $-p^2 n$ is also not a discriminant. Hence all terms in \eqref{eq:G1} vanish, and the identity holds trivially.

If $p = 2$, then $H^{(N_1,N_2)}(n) = 0$, and by Theorem~\ref{thm:Localexp},
$$
H^{(N_1,N_2)}(4n)
=
\frac{D_2(4n)}{C_2(4n)} \, H^{(2N_1,N_2)}(4n).
$$
Write $-4n = D_{4n} f_{4n}^2$, where $D_{4n}$ is a fundamental discriminant. We claim that $4 \mid D_{4n}$. Indeed, otherwise we could write
$$
-n = D_{4n} \left(\frac{f_{4n}}{2}\right)^2,
$$
which contradicts the assumption that $-n$ is not a discriminant. It follows that
$$
D_2(4n) = C_2(4n) = 1,
$$
and therefore \eqref{eq:G1} also holds in this case.

\medskip

\textbf{Case 2:} Assume $-n$ is a discriminant (i.e. $n \equiv 0,3 \mod 4$).

In this case, we have
$$
-p^2 n = D_n (p f_n)^2.
$$
Moreover, for any prime $q \neq p$, the $q$-parts of $f_{p^2 n}$ and $f_n$ coincide. Consequently,
$$
A_q(p^2 n)=A_q(n), \quad
B_q(p^2 n)=B_q(n), \quad
C_q(p^2 n)=C_q(n), \quad
D_q(p^2 n)=D_q(n)
$$
for all $q \neq p$.

Using Proposition~\ref{thm:Localexp} and the assumption $(p, N_1N_2)=1$, we obtain
\begin{align*}
H^{(N_1,N_2)}(p^2 n)
&=
L(0,\chi_{D_n})
\prod_{\substack{q\nmid N_1N_2}} D_q(p^2 n)
\prod_{q\mid N_1} C_q(p^2 n)
\prod_{q\mid N_2} B_q(p^2 n) \\
&=
\left(
L(0,\chi_{D_n})
\prod_{\substack{q\nmid N_1N_2\\ q\neq p}} D_q(n)
\prod_{q\mid N_1} C_q(n)
\prod_{q\mid N_2} B_q(n)
\right)\cdot
D_p(p^2 n),
\end{align*}
\begin{align*}
    H^{(N_1,pN_2)}(p^2 n)
&=
L(0,\chi_{D_n})
\prod_{\substack{q\nmid pN_1N_2}} D_q(p^2 n)
\prod_{q\mid pN_1} C_q(p^2 n)
\prod_{q\mid N_2} B_q(p^2 n)\\
&=
\left(
L(0,\chi_{D_n})
\prod_{\substack{q\nmid N_1N_2\\ q\neq p}} D_q(p^2 n)
\prod_{q\mid N_1} C_q(n)
\prod_{q\mid N_2} B_q(n)
\right)\cdot
C_p(n),
\end{align*}

and
\begin{align*}
H^{(N_1,N_2)}(n)
&=
\left(
L(0,\chi_{D_n})
\prod_{\substack{q \nmid N_1N_2\\ q\neq p}} D_q(n)
\prod_{q \mid N_1} C_q(n)
\prod_{q \mid N_2} B_q(n)
\right)\cdot
D_p(n).
\end{align*}

Therefore, Equation \eqref{eq:G1} is equivalent to
$$
D_p(p^2n)-pD_p(n)=C_p(n).
$$
We verify this identity directly:
\begin{align*}
D_p(p^2n)-pD_p(n)
&=
\frac{p+1}{p}A_p(p^2n)+\frac{1}{1-p}C_p(p^2n)-pD_p(n) \\
&=
\frac{p+1}{p}\,pA_p(n)+\frac{1}{1-p}C_p(n)-pD_p(n) \\
&=
p\left(\frac{p+1}{p}A_p(n)+\frac{1}{1-p}C_p(n)\right)-pD_p(n)+C_p(n) =C_p(n).
\end{align*}
Note that we use Lemma \ref{lem:linearlocalfactor} in the last line.
\end{proof}

\textbf{Remark:} From the above argument, we also observe that if $p \neq 2$, or if $p=2$ and $-n$ is a discriminant, then one may replace
$H^{(pN_1,N_2)}(p^2n)$ by $H^{(pN_1,N_2)}(n)$ in \ref{eq:G1}.

\begin{cor}\label{cor:4.9}
Let $p \neq 2$ be a prime, and let $N_2$ be an odd square-free integer such that $(N_2,2p)=1$. Then 
$$
H^{(p,N_2)}(4n)-2H^{(p,N_2)}(n)
=
H^{(2,N_2)}(4p^2n)-pH^{(2,N_2)}(4n).
$$
\end{cor}

\begin{proof}
This follows directly from Proposition~\ref{prop:G1} together with the preceding remark.
\end{proof}

\textbf{Remark:}
The genus identity of Berkovich and Jagy~\cite[Theorem~1.3]{BJ}
$$
r_3(p^2 n) - p r_3(n)
=
48 \sum_{f \in TG_{1,p}} \frac{R_f(n)}{|\mathrm{Aut}(f)|}
\;-\;
96 \sum_{f \in TG_{2,p}} \frac{R_f(n)}{|\mathrm{Aut}(f)|}.
$$
arises as a special case of the above proposition when $N_2=1$ and the equation \ref{eq:r3}; see also \cite[Section~7.2]{LuoZhou} for further details.

On the other hand, Beckwith and Mono \cite{Beckwith-Mono-2} using theta lift to justify the terminology ``generalized class numbers'' of $H_{m,N}(n)$ by the following result:
\begin{thm}[\cite{Beckwith-Mono-2} Theorem 1.4]\label{thm:BM4}
Let $p$ be prime.
If $n>0$ with $n \equiv 0,3 \mod 4$, then we have
$$
\sum_{Q \in Q_{p,-n}/\Gamma_0(p)} \frac{2}{|\Gamma_0(p)_Q|}
=
\frac{4(p+1)}{p} H_{1,p}(n)
-
\frac{2(p+1)}{p-1} H_{p,p}(n).
$$
\end{thm}

Note that the factor of $2$ on the left-hand side is absent in their formulation, since they compute stabilizer sizes in $\operatorname{PSL}_2(\mathbb{Z})$, while we work with classes under $\operatorname{SL}_2(\mathbb{Z})$.
Using Theorem \ref{thm:BM4} we can prove Theorem \ref{thm:Main4} mentioned at the beginning.

\begin{proof}[Proof of Theorem \ref{thm:Main4}]
    By Theorem \ref{thm:BM4}, \ref{thm:BM1.1}, and \ref{thm:PWtoLZ}, we can rewrite the sum as:

\begin{align*}
\sum_{Q \in \mathcal{Q}_{p,-n}/\Gamma_0(p)}
\frac{1}{|\Gamma_0(p)_Q|}
&=
4\left(
H(n)
-
\frac{1}{1-p} H_{p,p}(n)
\right)
-
\frac{2(p+1)}{p-1} H_{p,p}(n) \\
&= 4H(n) - 2H_{p,p}(n) \\
&= 4H(n) - 2H^{(p,1)}(n).
\end{align*}

Therefore we have 
\begin{align*}
12\sum_{Q \in \mathcal{Q}_{p,-4n}/\Gamma_0(p)}
&\frac{1}{|\Gamma_0(p)_Q|}
-
24
\sum_{Q \in \mathcal{Q}_{p,-n}/\Gamma_0(p)}
\frac{1}{|\Gamma_0(p)_Q|} \\
&=
4\cdot 12\cdot\bigl(H(4n) - 2H(n)\bigr)
-
2\cdot 12\cdot \bigl(H^{(p,1)}(4 n) - 2H^{(p,1)}(n)\bigr) \\
&=
4r_3(n)
-
2 \big(r_3(p^2 n) - p r_3(n)\big),
\end{align*}
where in the last equation, we use Corollary \ref{cor:4.9} with $N_2=1$. 
\end{proof}

\section*{Acknowledgments}

The author is deeply grateful to Olivia Beckwith for introducing this problem and for her continued guidance, as well as for many helpful comments and suggestions throughout this work. The author also thanks Andreas Mono, Ben Kane, Kathrin Bringmann, and Edna Jones for their valuable and insightful comments.


\end{document}